\documentclass[a4paper,12pt]{article}

\usepackage{amsmath,amsthm,amssymb,mathrsfs,latexsym,amsfonts}
\usepackage{graphicx,psfrag,epsfig}
\usepackage{bbm}
\usepackage{a4wide}
\usepackage[english]{babel}
\usepackage[latin1]{inputenc}
\usepackage{authblk}

\numberwithin{equation}{section}

\newtheorem{theorem}{Theorem}[section]

\newtheorem{proposition}[theorem]{Proposition}

\newtheorem{remark}[theorem]{\it Remark}

\newtheorem{Main}{Theorem}

\newcounter{paraga}[section]

\newcommand{\N}{\mathbb{N}}
\newcommand{\Z}{\mathbb{Z}}
\newcommand{\Q}{\mathbb{Q}}
\newcommand{\R}{\mathbb{R}}
\newcommand{\C}{\mathbb{C}}

\begin{document}

\def\MP{\,{<\hspace{-.5em}\cdot}\,}
\def\SP{\,{>\hspace{-.3em}\cdot}\,}
\def\PM{\,{\cdot\hspace{-.3em}<}\,}
\def\PS{\,{\cdot\hspace{-.3em}>}\,}
\def\EP{\,{=\hspace{-.2em}\cdot}\,}
\def\PP{\,{+\hspace{-.1em}\cdot}\,}
\def\PE{\,{\cdot\hspace{-.2em}=}\,}
\def\N{\mathbb N}
\def\C{\mathbb C}
\def\Q{\mathbb Q}
\def\R{\mathbb R}
\def\T{\mathbb T}
\def\A{\mathbb A}
\def\Z{\mathbb Z}
\def\demi{\frac{1}{2}}

\begin{titlepage}
  \title{\LARGE{\textbf{Positive measure of KAM tori for finitely differentiable Hamiltonians}}}
  \author{Abed Bounemoura \\
    CNRS - PSL Research University\\
    (Universit{\'e} Paris-Dauphine and Observatoire de Paris)}
\end{titlepage}

\maketitle

\begin{abstract}
Consider an integer $n \geq 2$ and real numbers $\tau>n-1$ and $l>2(\tau+1)$. Using ideas of Moser, Salamon proved that individual Diophantine tori persist for Hamiltonian systems which are of class $C^l$. Under the stronger assumption that the system is a $C^{l+\tau}$ perturbation of an analytic integrable system, P{\"o}schel proved the persistence of a set of positive measure of Diophantine tori. We improve the last result by showing it is sufficient for the perturbation to be of class $C^{l}$ and the integrable part to be of class $C^{l+2}$.
\end{abstract}

\section{Introduction and main results}\label{s1}

In this paper, we consider small perturbations of integrable Hamiltonian systems, which are defined by a Hamiltonian function of the form
\[ H(q,p)=h(p)+f(q,p), \quad (q,p) \in \T^{n} \times \R^n \]
where $n\geq 2$ is an integer and the norm $|f|=\varepsilon$ (in a suitable space of functions) is a small parameter. The Hamiltonian system associated to this function is then given by
\begin{equation*}
\begin{cases}
\dot{q}=\nabla_p H(q,p)=\nabla h (p)+ \nabla_p f(q,p),\\
\dot{p}=-\nabla_q H(q,p)=-\nabla_q f(q,p)
\end{cases}
\end{equation*}
where $\nabla_q H$ and $\nabla_p H$ denote the vector of partial derivatives with respect to $q=(q_1,\dots,q_n)$ and $p=(p_1,\dots,p_n)$. When $\varepsilon=0$, the system associated to $H=h$ is trivially integrable: all solutions are given by
\[ (q(t),p(t))=(q(0)+t \nabla h (q(0)),p(0))) \] 
and therefore, for each fixed $p \in \R^n$, letting $\omega=\nabla h(p) \in \R^n$, the sets $T_{\omega}=\T^n \times \{p\}$ are invariant tori on which the dynamics is given by the linear flow with frequency $\omega$. The integrable Hamiltonian $h$ is said to be non-degenerate on some ball $B \subseteq \R^n$ if the map $\nabla h : B \rightarrow \R^n$ is a diffeomorphism onto its image $\Omega=\nabla h(B)$. 

\medskip

It is a fundamental result of Kolmogorov that many of these unperturbed quasi-periodic tori persist under any sufficiently small perturbation (\cite{Kol54}), provided the system is real-analytic and the integrable part is non-degenerate. More precisely, Kolmogorov proved that given any vector $\omega \in \Omega$ satisfying the following Diophantine condition:
\begin{equation}\label{dioph}
|k\cdot \omega| \geq \gamma |k|^{-\tau}, \quad k=(k_1,\dots,k_n) \in \Z^n\setminus\{0\}, \quad |k|=|k_1|+\cdots+|k_n| \tag{$D_{\gamma,\tau}$}
\end{equation} 
where $\gamma>0$ and $\tau \geq n-1$ are fixed, the associated torus $T_{\omega}$ persist, being only slightly deformed into another Lagrangian real-analytic quasi-periodic torus $\mathcal{T}_\omega$ with the same frequency. Of course, there are uncountably many many vectors $\omega \in D_{\gamma,\tau}$, and thus the theorem of Kolmogorov gives uncountably many invariant tori. Even more, the set $D_{\gamma,\tau}$ does have positive Lebesgue measure when $\tau>n-1$; the measure of its complement in $\Omega$ (when the latter has a nice boundary) is of order $\gamma$, and hence one expect the set of quasi-periodic invariant tori to have positive Lebesgue measure in phase space. Unfortunately, this does not follow directly from the proof of Kolmogorov, but this was later showed to be correct by Arnold (\cite{Arn63a}) who introduced a different method to prove the theorem of Kolmogorov. Nowadays, the most common strategy to obtain positive measure is to show that the regularity of $\mathcal{T}_\omega$ with respect to $\omega$ is Lipschitz, as this immediately allows to transfer a positive measure set in the space of frequencies into a positive measure set in phase space. We refer to the nice survey~\cite{Pos01} for this Lipschitz dependence in the analytic case.

After Kolmogorov's breakthrough, an important contribution was made by Moser who proved that the Hamiltonian need not be real-analytic (see~\cite{Mos62} for the case of twist maps, which corresponds to an iso-energetic version of the theorem for $n=2$); it is sufficient for the Hamiltonian to be of finite but sufficiently high regularity (of course, the perturbed torus is then only finitely differentiable). Following an idea of Moser (\cite{Mos70}) and a work of P{\"o}schel (\cite{Pos80}), Salamon proved in~\cite{Sal04} that for the persistence of an indiviual tori $T_\omega$ with $\omega \in D_{\gamma,\tau}$, it is sufficient to require the system to be of class $C^l$, with $l>2(\tau+1)$: the torus is then of class $C^{\tau+1}$ and the dynamic on it is $C^1$-conjugated to the linear flow. The regularity of the perturbation can be mildly improved as was shown in~\cite{Alb}, and it may be possible to actually reach the value $l=2(\tau+1)$, but in any event the theorem cannot be true for $l<2(\tau+1)$ as was proved in~\cite{CW13}. Let us point out that for twist maps of the annulus, optimal regularity results follow from the work of Herman (\cite{Her86}). All those results concern the persistence of individual quasi-peridic tori. As for the persistence of a set of positive measure, after an initial result of Lazutkin again for twist maps (\cite{Laz73}) that required an excessive amount of differentiability, the most general result so far is due to P{\"o}schel. In~\cite{Pos82}, he proved the persistence of a set of positive measure under the assumption that the perturbation is of class $C^{l+\tau}$ and the integrable part is real-analytic. Actually, under those assumptions, he proved that the regularity with respect to $\omega$ is $C^1$ in the sense of Whitney (and if the perturbation is more regular, then one has more regularity with respect to $\omega$); this implies in particular Lipschitz dependence. However, the regularity assumptions in the work of P{\"o}schel are definitely stronger than those in the work of Salamon, as not only the perturbation is required to be of class $C^{l+\tau}$ instead of class $C^{l}$ but the integrable part is required to be real-analytic (such analyticity assumption is also present in~\cite{Mos70},~\cite{Pos80} and~\cite{Alb}).

\medskip

It is our purpose here to actually prove that we have persistence of a set of positive measure of quasi-periodic tori provided the perturbation is of class $C^{l}$, as in~\cite{Sal04}, and the integrable part is of class $C^{l+2}$, which is slightly stronger than the assumption in~\cite{Sal04} but still much better than the analyticity assumption of~\cite{Pos82} (we observe, in Remark~\ref{reg} below, that for a fixed $\omega$, we can actually assume $h$ to be of class $C^l$ and not necesarily integrable, and one could recover~\cite{Sal04}).

It is important to point out, however, that we essentially do not improve the main technical result of~\cite{Pos82}. To explain this, let us recall that one can look at the perturbed invariant torus $\mathcal{T}_\omega$ in at least two way: either as the image of an embedding $\Psi_\omega : \T^n \rightarrow \T^n \times B$ into phase space, which moreover conjugates the restricted dynamics to a linear flow, or as the graph of a function $\Gamma_\omega : \T^n \rightarrow B$ defined on the configuration space. The main observation we will use is that the graph is usually more regular than the embedding. This is not new, as for instance in~\cite{Sal04} the embedding is only $C^1$ while the graph is $C^{\tau+1}$, and this is also not surprising. Indeed, the graph $\Gamma_\omega$ only gives the invariant torus, whereas the embedding $\Psi_\omega$ also encodes the dynamical information, as it conjugates the restricted dynamics to a linear flow on the torus, so there is a priori no reason for these two objects to have the same regularity. It is well-known that it is hard to actually construct the invariant graph without prescribing the dynamic on it. Yet we will be able to use this basic observation to show that under our regularity assumption ($f$ is $C^l$ and $h$ is $C^{l+2}$ for $l>2(\tau+1)$), we can construct $\Psi_\omega$ and $\Gamma_\omega$ in such a way that $\Gamma_\omega$ is Lipschitz with respect to $\omega$, without knowing whether this is the case for $\Psi_\omega$. Now P{\"o}schel proved that $\Psi_\omega$ is Lipschitz with respect to $\omega$, provided $f$ is $C^{l+\tau}$ and $h$ is analytic; we could recover (and slightly extend this result) simply by replacing $l$ by $l+\tau$ in our assumption, but clearly this does not improve in any way the measure estimate in phase space of the set of perturbed invariant tori. To summarize this discussion, P{\"o}schel proves that not only the torus $\mathcal{T}_\omega$ but also the restricted dynamics is Lipschitz with respect to $\omega$, whereas we only prove, under weaker and almost optimal regularity assumption (at least concerning the perturbation), the first assertion, which is the one needed to have a set of positive measure in phase space.

\medskip

We now state more precisely our result, and consider
\begin{equation}\label{Ham}
\begin{cases}
H : \T^n \times B \rightarrow \R, \\
H(q,p)=h(p)+f(p,q), \\
h \; \mbox{non-degenerate}.
\end{cases}
\tag{$*$}
\end{equation}
Recall that $\Omega=\nabla h(B)$, we let $\partial \Omega$ its boundary and for fixed constants $\gamma >0$ and $\tau \geq n-1$, we define the following set of Diophantine vectors
\begin{equation*}
\Omega_{\gamma,\tau}=\{\omega \in D_{\gamma,\tau} \cap \Omega \; | \; \mathrm{d}(\omega,\partial\Omega) \geq \gamma\}.
\end{equation*}
As we already explained, when $f=0$, the phase space is trivially foliated by invariant quasi-periodic tori $T_\omega$ which are invariant by $H=h$; since $h$ is non-degenerate, $\nabla h$ has an inverse $(\nabla h)^{-1}$ and the Lipschitz constant of $T_\omega$ with respect to $\omega$, that we shall denote by $\mathrm{Lip}(T)$, is nothing but the Lipschitz constant $\mathrm{Lip}((\nabla h)^{-1})$ of $(\nabla h)^{-1}$. For simplicity, we shall denote the $C^l$ norms of functions by $|\;.\;|_l$, without referring to their domain of definition which should be clear from the context.

\begin{Main}\label{thm1}
Let $H$ be as in~\eqref{Ham} of class $C^l$ with $l>2(\tau+1)$, and assume that
\begin{equation}\label{seuila}
\epsilon=|f|_l \leq c \gamma^2 
\end{equation}
for some small constant $c>0$ which depends only on $n$, $\tau$, $l$ and the norms $|h|_{l+2}$ and $|(\nabla h)^{-1}|_{l}$. Then there exists a set
\[ \mathcal{K}_{\gamma,\tau}=\bigcup_{\omega \in \Omega_{\gamma,\tau}}\mathcal{T}_\omega \subseteq \T^n \times B  \]
where each $\mathcal{T}_\omega$ is an invariant Lagrangian torus of class $C^{\tau+1}$, Lipschitz with respect to $\omega$, and on which the Hamiltonian flow is $C^{1}$-conjugated to the linear flow with frequency $\omega$. Moreover, as $\epsilon$ goes to zero, $\mathcal{T}_\omega$ converges to $T_\omega$ in the $C^{\tau+1}$ topology and $\mathrm{Lip}(\mathcal{T}_\omega)$ converges to $\mathrm{Lip}(T_\omega)$. Finally, we have the measure estimate
\[ \mathrm{Leb}(\T^n \times B  \setminus \mathcal{K}_{\gamma,\tau}) \leq C\gamma^2 \]
provided $\partial \Omega$ is piecewise smooth, where $ \mathrm{Leb}$ denotes the Lebesgue measure and $C>0$ is a large constant. 
\end{Main}

The last part of the statement, concerning the measure estimate, is a well-known consequence of the first part, so we shall not give details (see~\cite{Pos01} for instance). Let us point out that we could have proved that $\mathcal{T}_\omega$ is actually $C^1$ in the sense of Whitney with respect to $\omega$, but we have chosen not to do so (in order not to introduce this notion, as well as anisotropic differentiability, see~\cite{Pos82} for instance).

\section{KAM theorem with parameters}\label{s2}

In this section, following~\cite{Pos01} and~\cite{Pop04}, we will deduce Theorem~\ref{thm1} from a KAM theorem in which the frequencies are taken as independent parameters.

Let us consider the Hamiltonian $H=h+f$ as in~\eqref{Ham}, with $\Omega=\nabla h(B)$ and where we recall that $\Omega_{\gamma,\tau}$ is the set of $(\gamma,\tau)$-Diophantine vectors in $\Omega$ having a distance at least $\gamma$ from the boundary $\partial \Omega$. Now choose $\tilde{\Omega} \subseteq \Omega$ a neighborhood of $\Omega_{\gamma,\tau}$ such that both the distance of $\Omega_{\gamma,\tau}$ to $\partial \tilde{\Omega}$ and $\tilde{\Omega}$ to $\partial \Omega$ is at least $\gamma/2$. Since $\nabla h : B \rightarrow \Omega$ is a diffeomorphism, we can define $\tilde{B}=(\nabla h)^{-1}(\tilde{\Omega})$. For $p_0 \in \tilde{B}$, we expand $h$ in a sufficiently small ball of radius $\rho>0$ around $p_0$: writing $p=p_0+I$ for $I$ in the ball $B_\rho$ of radius $\rho$ centered at zero, we have that $p \in B$ provided
\begin{equation}\label{seuilaa}
\rho \leq (2|h|_{l+2})^{-1}\gamma
\end{equation}
and under this assumption, we can write
\[ h(p)=h(p_0)+ \nabla h(p_0)\cdot I + \int_{0}^{1}(1-t)\nabla^2 h(p_0 +tI)I\cdot I dt. \]
As $\nabla h : \tilde{B} \rightarrow \tilde{\Omega}$ is a diffeomorphism, instead of $p_0$ we can use $\omega=\nabla h(p_0)$ as a new variable, and we write
\[ h(p)=e(\omega)+ \omega\cdot I + P_h(I,\omega) \] 
with 
\begin{equation}\label{Ph}
e(\omega)=h((\nabla h)^{-1}\omega), \quad P_h(I,\omega)=\int_{0}^{1}(1-t)\nabla^2 h((\nabla h)^{-1}\omega +tI)I\cdot Idt.
\end{equation}
Letting $\theta=q$ and 
\[ P_f(\theta,I,\omega)=f(q,p)=f(\theta,p_0+I)= f(\theta,(\nabla h)^{-1}\omega+I), \]
we eventually arrive at
\begin{equation}\label{param}
h(q,p)=H(\theta,I,\omega)=e(\omega)+ \omega\cdot I + P_h(I,\omega)+P_f(\theta,I,\omega)
\end{equation}
where we recall that $I$ varies in a small ball of radius $\rho>0$ around zero. Since $h$ is of class $C^{l+2}$ and $f$ is of class $C^l$, obviously $P_h$ and $P_f$ are of class $C^l$. Moreover, if the $C^l$ norm of $f$ is small then so is the $C^l$ norm of $P_f$ but unfortunately this is not necessarily the case for the $C^l$ norm of $P_h$, not matter how small we choose $\rho$. We will therefore rescale the variable $I$ (another essentially equivalent way to deal with this issue is to use a weighted norm as in~\cite{Pop04}): we consider 
\begin{equation}\label{param2}
\begin{cases}
\rho^{-1}H(\theta,\rho I,\omega)=e_\rho(\omega)+ \omega\cdot I + P_\rho(\theta,I,\omega), \\
e_\rho(\omega)=\rho^{-1}e(\omega), \quad P_\rho(\theta,I,\omega)=\rho^{-1}P_h(\rho I,\omega)+\rho^{-1}P_f(\theta,\rho I,\omega)
\end{cases}
\end{equation}
where now $I$ varies in the unit ball $B_1$ around the origin, and it is easy to observe that the $C^l$ norm of $P_\rho$ will be small provided we choose $\rho$ small. As a side remark, the term $e_\rho$ will be large, but its size is irrelevant to our problem (it simply does not appear in the Hamiltonian vector field associated to $H$). 

\begin{remark}\label{reg}
Observe that we can also write
\[ P_h(I,\omega)=\rho^{-1}h((\nabla h)^{-1}\omega+\rho I)-\rho^{-1}h((\nabla h)^{-1}\omega)-\omega\cdot I. \]
So for a fixed $\omega$, looking at $P_h$ as a function of $I$ only, the above expression shows that it is sufficient for $h$ to be of class $C^l$ for $P_h$ to be of class $C^l$ and with a $C^l$ norm of order $\rho$. Moreover, for a fixed $\omega$, clearly $h$ needs not be integrable (it suffices to consider $h(q,p)$ with $h(q,0)$ constant and $\nabla_p h(q,0)=\omega$, as the fact that $P_h$ is independent of $\theta$ will not play any role in the sequel). For a variable $\omega$, we need at least $h$ to be of class $C^{l+1}$ in order to have $P_h$ of class $C^l$, and we required $h$ to be of class $C^{l+2}$ so that the $C^l$ norm of $P_h$ is of order $\rho$. 
\end{remark}

This discussion leads us to consider the following abstract Hamiltonian:
\begin{equation}\label{Ham2}
\begin{cases}
H : \T^n \times B_1 \times \tilde{\Omega} \rightarrow \R, \\
H(\theta,I,\omega)=e(\omega)+ \omega\cdot I + P(\theta,I,\omega).
\end{cases}\tag{$**$}
\end{equation}
We shall consider $\omega$ as a parameter, so when convenient we will write $H(\theta,I,\omega)=H_\omega(\theta,I)$. Theorem~\ref{thm1} will be obtained from the following statement.

\begin{Main}\label{thm2}
Let $H$ be as in~\eqref{Ham2} of class $C^l$ with $l>2(\tau+1)$. There exists a small constant $c>0$ which depends only on $n,\tau$ and $l$ such that if 
\begin{equation}\label{seuilb}
\varepsilon=|P|_l \leq \tilde{c} \gamma 
\end{equation} 
then the following holds true. There exists a continuous map $\varphi : \Omega_{\gamma,\tau} \rightarrow \tilde{\Omega}$, and for each $\omega \in \Omega_{\gamma,\tau}$, a $C^1$ map $\Psi_\omega=(U_\omega,G_\omega) : \T^n \rightarrow \T^n \times B_1$ such that $\Gamma_\omega=G_\omega \circ U_\omega^{-1} : \T^n \rightarrow B_1$ is of class $C^{\tau+1}$ and:

\noindent
$(1)$ The set
\[ \Psi_\omega(\T^n)=\{(U_\omega(\theta),G_\omega(\theta) \; | \; \theta \in \T^n)\} =\{(\theta,\Gamma_\omega(\theta) \; | \; \theta \in \T^n)\} \]
is an embedded Lagrangian torus invariant by the flow of $H_{\varphi(\omega)}$ with 
\[ X_{H_{\varphi(\omega)}} \circ \Psi_\omega=\nabla \Psi_\omega\cdot \omega \]
and moreover, $|U_\omega-\mathrm{Id}|_1$, $|G_\omega|_1$ and $|\Gamma_\omega|_{\tau+1}$ converge to zero as $\varepsilon$ goes to zero; 

\noindent
$(2)$ The map $\varphi$ and $\Gamma$ are Lipschitz in $\omega$ and moreover, $\mathrm{Lip}(\varphi-\mathrm{Id})$ and $\mathrm{Lip}(\Gamma)$ converge to zero as $\varepsilon$ goes to zero.
\end{Main}

Theorem~\ref{thm2} will be proved in the next section; here we will show how this easily implies Theorem~\ref{thm1}.

\begin{proof}[Proof of Theorem~\ref{thm1}]
Let $H$ be as in~\eqref{Ham}, and choose $\rho=\sqrt{\epsilon}$. With this choice, it follows from~\eqref{Ph},~\eqref{param} and~\eqref{param2} that $H$ can be written as in~\eqref{Ham2} with $\varepsilon=\tilde{C}\sqrt{\epsilon}$, with a large constant $\tilde{C}$ that depends only $n$, $l$, the $C^{l+2}$ norm of $h$ and the $C^{l}$ norm of $(\nabla h)^{-1}$. In view of the assumption~\eqref{seuila} of Theorem~\ref{thm1}, both~\eqref{seuilaa} and the assumption~\eqref{seuilb} of Theorem~\ref{thm2} are satisfied and it suffices to define
\[ \mathcal{T}_\omega=\{(\theta,\Gamma_\omega(\theta)+(\nabla h)^{-1}(\varphi(\omega))) \; | \; \theta \in \T^n\}  \] 
so that the conclusions of Theorem~\ref{thm1} follow from those of Theorem~\ref{thm2}. 
\end{proof}

\section{Proof of Theorem~\ref{thm2}}\label{s3}

Before starting the proof of Theorem~\ref{thm2}, we observe, as in~\cite{Pos82}, that by scaling the frequency variables $\omega$, it is enough to prove the statement for a normalized value of $\gamma$, so without loss of generality, we may assume that $\gamma=1$ in the sequel. As before, observe that the term $e$ gets transform into $\gamma^{-1}e$, but its size is of no importance. Let us also introduce some notations. We set $\nu=\tau+1$, our regularity assumption then reads $l>2\nu$ and thus we can find real numbers $\lambda$ and $\chi$ such that
\begin{equation}\label{numbers}
l=\lambda+\nu+\chi, \quad \chi>0, \quad \kappa=\lambda-\nu-\chi=2\lambda-l>0.
\end{equation}
For a later purpose, associated to $\kappa=2\lambda-l$ we defined above we introduce the real number $0<\delta<1$ defined by
\begin{equation}\label{delta}
\delta=6^{-1/\kappa}.
\end{equation}
In this paper, we do not pay attention to how constants depend on the dimension $n$, the  Diophantine exponent $\tau$ and the regularity $l$, as they are all fixed. Hence from now on, we shall use a notation of~\cite{Pos01} and write
$$u \MP v \quad (\mbox{respectively } u \PM v)$$
if, for some constant $C\geq 1$ depending only on $n$, $\tau$ and $l$ we have $u\leq Cv$ (respectively $Cu \leq v$). We will also use the notation $u \EP v$ and $u \PE v$ which is defined in a similar way. 

\subsection{Analytic smoothing}

In this section, we will approximate our perturbation $P$ in~\eqref{Ham2} by a sequence of analytic perturbations $P_j$, $j \in \N$, defined on suitable complex domains. Using bump functions, we first extend, keeping the same notations, $P$, which is initially defined on $\T^n \times B_1 \times \tilde{\Omega}$, as a function defined on $\T^n \times \R^n \times \R^n$ with support in $\T^n \times B_2 \times \Omega$. This only changes the $C^l$ norm of $P$ by a multiplicative constant which depends only on $n$ and $l$.

Given $0< u_0 \leq 1$, consider the geometric sequence $u_j=u_0\delta^{j}$ where $\delta$ is defined in~\eqref{delta}. Associated to this sequence we define a decreasing sequence of complex domains
\begin{equation}\label{complexdom} 
\mathcal{U}_j= \{(\theta,I,\omega) \in \C^n/\Z^n \times \C^n \times \C^n \; | \; \mathrm{Re}(\theta,I,\omega) \in \T^n \times B_2 \times \Omega, \; |\mathrm{Im}(\theta,I,\omega)| \leq u_j 
 \}
\end{equation}
where $|\,.\,|$ stands, once and for all, for the supremum norm of vectors. The supremum norm of a real-analytic (vector-valued) function $F : \mathcal{U} \rightarrow \C^p$, $p \geq 1$, will be denoted by
\[ |F|_{\mathcal{U}}=\sup_{z \in \mathcal{U}} |F(z)|. \]
We have the following approximation result. 

\begin{proposition}\label{smoothing}
Let $P$ be as in~\eqref{Ham2}, and let $u_j=u_0\delta^{j}$ for $j \in \N$ with $0< u_0 \leq 1$. There exists a sequence of analytic functions $P_j$ defined on $\mathcal{U}_j$, $j \in \N$, with 
\[  |P_0|_{\mathcal{U}_0} \MP |P|_l, \quad |P_{j+1}-P_{j}|_{\mathcal{U}_{j+1}} \MP u_{j}^l|P|_l, \quad \lim_{j \rightarrow +\infty}|P_j-P|_1=0. \]
\end{proposition}

This proposition is well-known, we refer to~\cite{Zeh75} or~\cite{Sal04} for a proof. It is important to observe that the implicit constant in the above statement do not depend on $u_0$, which has yet to be chosen.

\subsection{Analytic KAM step}

In this section, we state an elementary step of an analytic KAM theorem with parameters, following the classical exposition of P{\"o}schel (\cite{Pos01}) but with some modifications taken from~\cite{Rus01}. Given $s,r,h$ real numbers such that $0 \leq s,r,h \leq 1$, we let
\[
\begin{cases}
\mathcal{V}_s= \{ \theta \in \C^n / \Z^n \; | \; |\mathrm{Im}(\theta)|<s \}, \\ 
\mathcal{V}_{r}=\{ I \in \C^n \; | \; |I|<r \}, \\
\mathcal{V}_h=\{ \omega \in \C^n \; | \; |\omega-\Omega_1|<h \}
\end{cases}
\]
and we define 
\[ \mathcal{V}_{s,r,h}=\mathcal{V}_s\times \mathcal{V}_r \times \mathcal{V}_h \]
which is a complex neighborhood of respectively $\T^n \times \{0\} \times \Omega_1$, and where $\Omega_1$ is a set of $(1,\tau)$-Diophantine vectors having a distance at least $1$ from the boundary of $\Omega$. Consider a function $H$, which is real-analytic on $\mathcal{V}_{s,r,h}$, of the form 
\begin{equation}\label{Ham3}
\begin{cases}
H(\theta,I,\omega)=N(I,\omega)+R(\theta,I,\omega)=e(\omega)+\omega \cdot I+R(\theta,I,\omega), \\
|R|_{s,r,h}=|R|_{\mathcal{V}_{s,r,h}} <+\infty.
\end{cases}
\tag{$***$}
\end{equation}
Again, the function $H=N+R$ should be considered as a real-analytic Hamiltonian on $\mathcal{V}_{s,r}=\mathcal{V}_s \times \mathcal{V}_r$, depending analytically on a parameter $\omega \in V_h$. To such Hamiltonians, we will apply transformations of the form
\[ \mathcal{F}=(\Phi,\varphi): (\theta,I,\omega) \mapsto (\Phi(\theta,I,\omega),\varphi(\omega))=(\Phi_\omega(\theta,I),\varphi(\omega)) \]
which consist of a parameter-depending change of coordinates $\Phi_\omega$ and a change of parameters $\varphi$. Moreover, setting $\mathcal{V}_{s,h}=\mathcal{V}_s \times \mathcal{V}_h$, our change of coordinates will be of
the form
\[\Phi(\theta,I,\omega)=(U(\theta,\omega),V(\theta,I,\omega))=(\theta+E(\theta,\omega),I+F(\theta,\omega)\cdot
I+G(\theta,\omega)) \] with
\[ E : \mathcal{V}_{s,h} \rightarrow \C^n, \quad F : \mathcal{V}_{s,h} 
\rightarrow M_n(\C), \quad G : \mathcal{V}_{s,h} \rightarrow \C^n \]
and for each fixed parameter $\omega$, $\Phi_\omega$ will be
symplectic. The composition of such transformations
\[ \mathcal{F}=(\Phi,\varphi)=(U,V,\varphi)=(E,F,G,\varphi) \] 
is again a transformation of the same form, and we shall denote by $\mathcal{G}$ the groupoid of such transformations. For functions defined on $\mathcal{V}_{s,h}$, we will denote by $\nabla_\theta$ (respectively $\nabla_\omega$) the vector of partial derivatives with respect to $\theta$ (respectively with respect to $\omega$). We have the following proposition.

\begin{proposition}\label{Pos}
Let $H=N+R$ be as in~\eqref{Ham3}, and suppose that $|R|_{s,r,h} \leq \varepsilon$ with
\begin{equation}\label{seuil}
\begin{cases}
\varepsilon \PM \eta^2 r \sigma^{\nu}, \\
\varepsilon \PM h r, \\
h \leq (2K^{\nu})^{-1}, \quad K=n\sigma^{-1}\log(\eta^{-2})
\end{cases}
\end{equation}
where $0<\eta<1/4$ and $0<\sigma<s/5$. Then there exists a transformation 
\[ \mathcal{F}=(\Phi,\varphi)=(U,V,\varphi) : \mathcal{V}_{s-4\sigma,\eta r,h/4} \rightarrow  \mathcal{V}_{s,r,h}, \quad U^{-1} : \mathcal{V}_{s-5\sigma,h} \rightarrow \mathcal{V}_{s-4\sigma}\] 
that belongs to $\mathcal{G}$, such that,  letting $|\;.\;|^*$ the supremum norm on the domain $\mathcal{V}_{s-4\sigma,\eta r,h/4}$ and $|\;.\;|$ the supremum norm on the domain $\mathcal{V}_{s-5\sigma,\eta r,h/4}$, we have 
\begin{equation}\label{estim1}
H \circ \mathcal{F}=N^++R^+, \quad |R^+| \leq 3\eta^2\varepsilon
\end{equation}
and  
\begin{equation}\label{estim2}
\begin{cases}
|E|^* \MP \varepsilon(r\sigma^\tau)^{-1}, \quad |\nabla_\theta E| \MP \varepsilon(r\sigma^\nu)^{-1}, \quad |\nabla_\omega E| \MP \varepsilon(h r\sigma^\tau)^{-1},\\
|F|^* \MP \varepsilon(r\sigma^\nu)^{-1}, \quad |\nabla_\theta F| \MP \varepsilon(r\sigma^{\nu+1})^{-1}, \quad |\nabla_\omega F| \MP \varepsilon(h r\sigma^\nu)^{-1}, \\
|G|^* \MP \varepsilon(\sigma^\nu)^{-1}, \quad |\nabla_\theta G| \MP \varepsilon(\sigma^{\nu+1})^{-1}, \quad |\nabla_\omega G| \MP \varepsilon(h \sigma^\nu)^{-1},  \\
|\varphi-\mathrm{Id}| \MP \varepsilon r^{-1}, \quad |\nabla \varphi-\mathrm{Id}| \MP  \varepsilon (h r)^{-1}.
\end{cases}
\end{equation}

\end{proposition}

The above proposition is the KAM step of~\cite{Pos01}, up to some differences we now describe. The main difference is that in the latter reference, instead of~\eqref{seuil} the following conditions are imposed (comparing notations, we have to put $P=R$ and $\alpha=1$):
\begin{equation}\label{seuilP}
\begin{cases}
\varepsilon \PM \eta r \sigma^{\nu}, \\
\varepsilon \PM h r, \\
h \leq (2K^{\nu})^{-1}
\end{cases}
\end{equation}
with a free parameter $K$, leading to the following estimate
\begin{equation}\label{estimP}
|R^+| \MP (\varepsilon(r\sigma^\nu)^{-1}+\eta^2+K^ne^{-K\sigma})\varepsilon.
\end{equation}
instead of~\eqref{estim1}. The last two terms in the estimate~\eqref{estimP} comes from the approximation of $R$ by a Hamiltonian $\hat{R}$ which is affine in $I$ and a trigonometric polynomial in $\theta$ (of order $K$); to obtain such an approximation, in~\cite{Pos01} the author simply truncates the Taylor expansion in $I$ and the Fourier expansion in $\theta$ to obtain the following approximation error
\[ |R-\hat{R}|_{s-\sigma,2\eta r, h} \MP (\eta^2+K^ne^{-K\sigma}).  \]
Yet we can use a more refined approximation result, namely Theorem $7.2$ of~\cite{Rus01} (choosing, in the latter reference, $\beta_1=\cdots=\beta_n=1/2$ and $\delta^{1/2}=2\eta$ for $\delta \leq 1/4$); with the choice of $K$ as in~\eqref{seuil}, this gives another approximation $\tilde{R}$ (which is nothing but a weighted truncation, both in the Taylor and Fourier series) and a simpler error
\[ |R-\tilde{R}|_{s-\sigma,2\eta r, h} \leq 2\eta^2.  \]
As for the first term in the estimate~\eqref{estimP}, it can be easily bounded by $\eta^2\varepsilon$ in view of the first part of~\eqref{seuil} which is stronger than the first part of~\eqref{seuilP} required in~\cite{Pos01}. Let us point out that we will use Proposition~\ref{Pos} in an iterative scheme which will not be super-linear as $\eta$ will be chosen to be a small but fixed constant, and thus having an estimate of the form~\eqref{estim1} will be more convenient for us.  

There are also minor differences with the statement in~\cite{Pos01}. The first one is that we observe that the coordinates transformation is actually defined on a domain $\mathcal{V}_{s-4\sigma,\eta r,h/4}$ which is slightly larger than the domain $\mathcal{V}_{s-5\sigma,\eta r,h/4}$ on which the new perturbation $R^+$ is estimated. The angle component of the transformation $U(\theta,\omega)=\theta+E(\theta,\omega)$ actually sends $\mathcal{V}_{s-4\sigma,h}$ into $\mathcal{V}_{s-3\sigma}$ so that its inverse is well-defined on $\mathcal{V}_{s-5\sigma,h}$ and maps it into $\mathcal{V}_{s-4\sigma}$ as stated. This simple observation will be important later, as this will imply an estimate of the form 
\[ |G \circ U^{-1}| \leq |G|^* \]
which will ultimately lead to an invariant graph which is more regular than the invariant embedding. The second one is that we expressed the estimates~\eqref{estim2} in a different, more cumbersome, way than it is in~\cite{Pos01} where weighted matrices are used. However, even though the use of weighted matrices is more elegant, they do not take into account the structure of the transformation which will be important in the convergence proof of Theorem~\ref{thm2} (see the comment after Proposition~\ref{convsmoothing} below).  

In the proof of Theorem~\ref{thm2}, Proposition~\ref{Pos} will be applied infinitely many times with sequences $0 < s_j \leq 1$, $0 <r_j \leq 1$ and $0 < h_j \leq 1$ that we shall now define. First for $0<s_0\leq 1$ to be chosen later (in the proof of Proposition~\ref{iterative}), we simply put
\begin{equation}\label{sj}
s_j=s_0\delta^{j}, \quad j \in \N
\end{equation} 
where $0<\delta<1$ is the number defined in~\eqref{delta}. Observe that this implies that Proposition~\ref{Pos} will be applied at each step with $\sigma=\sigma_j$ defined by 
\begin{equation}\label{sigma}
\sigma_j=(1-\delta)s_j/5
\end{equation} 
so that $s_{j+1}=s_j-5\sigma_j$. Let us also define
\begin{equation}\label{sj+}
s_{j+1}^*=s_j-4\sigma_j>s_j.
\end{equation} 
Then recall from~\eqref{numbers} that we have
\[ l=\lambda+\nu+\chi, \quad \chi>0, \quad \lambda>\nu+\chi. \]
We now define
\begin{equation}\label{rj}
r_j=s_j^\lambda=s_0^\lambda\delta^{j \lambda}=r_0\eta^{j}, \quad \eta=\delta^{\lambda}, \quad j \in \N.
\end{equation} 
Our choice of $\delta$ in~\eqref{delta} was made in order to have
\begin{equation}\label{eta}
3\eta^2=3\delta^{2\lambda}=3\delta^{\kappa}\delta^l=\delta^l/2.
\end{equation}
Finally, in view of the definition of $\sigma_j$ in~\eqref{sigma} and the choice of $\eta$ above, we define, on account of the last condition of~\eqref{seuil},
\begin{equation}\label{hj}
h_j=\bar{h}s_j^\nu, \quad \bar{h}=2^{-1}(1-\delta)^\nu(5n\lambda\log(\delta^{-2}))^{-\nu} \quad j \in \N.
\end{equation} 
Let us further denote $\mathcal{V}_j=\mathcal{V}_{s_j,r_j,h_j}$, $\mathcal{V}_j^*=\mathcal{V}_{s_j^*,r_j,h_j}$ and $|\;.\;|_j$ and $|\;.\;|_j^*$ the supremum norm on those domains. The following statement is a direct consequence of Proposition~\ref{Pos} with our choices of sequences, since $h_{j+1}\leq h_j/4$ for $j \in \N$, and the equality~\eqref{eta}.

\begin{proposition}\label{Posj}
Let $H_j=N_j+R_j$ be as in~\eqref{Ham3}, and suppose that $|R_j|_{j} \leq \varepsilon_j$ for $j \in \N$ with
\begin{equation}\label{seuilj}
\varepsilon_j \MP s_j^{\lambda+\nu}.
\end{equation}
Then there exists a transformation 
\[ \mathcal{F}_{j+1}=(\Phi_{j+1},\varphi_{j+1})=(U_{j+1},V_{j+1},\varphi_{j+1}) : \mathcal{V}_{j+1}^* \rightarrow  \mathcal{V}_j, \quad U_{j+1}^{-1}: \mathcal{V}_{j+1} \rightarrow \mathcal{V}_{j+1}^*\] 
that belongs to $\mathcal{G}$, such that,  
\begin{equation}\label{estim1j}
H_j \circ \mathcal{F}_{j+1}=N_j^++R_j^+, \quad |R_j^+| \leq \delta^l\varepsilon_j/2
\end{equation}
and  
\begin{equation}\label{estim2j}
\begin{cases}
|E_{j+1}|_{j+1}^* \MP \varepsilon_js_j^{-\lambda-\tau}, \quad |\nabla_\theta E_{j+1}|_{j+1} \MP \varepsilon_js_j^{-\lambda-\nu}, \quad |\nabla_\omega E_{j+1}|_{j+1} \MP \varepsilon_js_j^{-\lambda-\tau-\nu},\\
|F_{j+1}|_{j+1}^* \MP \varepsilon_js_j^{-\lambda-\nu}, \quad |\nabla_\theta F_{j+1}|_{j+1} \MP \varepsilon_js_j^{-\lambda-\nu-1}, \quad |\nabla_\omega F_{j+1}|_{j+1} \MP \varepsilon_js_j^{-\lambda-2\nu}, \\
|G_{j+1}|_{j+1}^* \MP \varepsilon_js_j^{-\nu}, \quad |\nabla_\theta G_{j+1}|_{j+1} \MP \varepsilon_js_j^{-\nu-1}, \quad |\nabla_\omega G_{j+1}|_{j+1} \MP \varepsilon_js_j^{-2\nu},  \\
|\varphi_{j+1}-\mathrm{Id}|_{j+1} \MP \varepsilon_j s_j^{-\lambda}, \quad |\nabla \varphi_{j+1}-\mathrm{Id}|_{j+1} \MP \varepsilon_j s_j^{-\lambda-\nu}.
\end{cases}
\end{equation}
\end{proposition}
  
\subsection{Iteration and convergence}

We will now combine Proposition~\ref{smoothing} and Proposition~\ref{Posj} into the following iterative Proposition which will be the main ingredient in the proof of Theorem~\ref{thm2}. Yet we still have to choose $u_0$ in Proposition~\ref{smoothing} and $s_0$ in Proposition~\ref{Posj}. We set
\begin{equation}\label{s0}
u_0=\delta^{-1} s_0, \quad s_0^l \EP \varepsilon
\end{equation}  
where we recall that $|P|_l \leq \varepsilon$ in~\eqref{Ham2}, and the above implicit constant is nothing but the implicit constant that appears in Proposition~\ref{smoothing}.

\begin{proposition}\label{iterative}
Let $H$ be as in~\eqref{Ham2} of class $C^l$ with $l>2\nu$, and consider the sequence $P_j$ of real-analytic Hamiltonians associated to $P$ given by Proposition~\ref{smoothing}. Then for $\varepsilon$ sufficiently small, the following holds true. For each $j \in \N$, there exists a normal form $N_{j}$, with $N_0=N$, and a transformation 
\begin{equation}\label{fhaut}
\mathcal{F}^{j+1}=(\Phi^{j+1},\varphi^{j+1})=(U^{j+1},V^{j+1},\varphi^{j+1}) : \mathcal{V}_{j+1} \rightarrow  \mathcal{U}_{j+1}, \quad \mathcal{F}^0=\mathrm{Id}
\end{equation} 
that belongs to $\mathcal{G}$, such that 
\begin{equation}\label{estim1iter}
(N+P_j) \circ \mathcal{F}^{j+1}=N_{j+1}+R_{j+1}, \quad |R_{j+1}|_{j+1} \leq s_{j+1}^l/2.
\end{equation}
Moreover, we have $\mathcal{F}^{j+1}=\mathcal{F}^{j} \circ \mathcal{F}_{j+1}$ with 
\begin{equation}\label{fbas}
\mathcal{F}_{j+1}=(\Phi_{j+1},\varphi_{j+1})=(U_{j+1},V_{j+1},\varphi_{j+1}) : \mathcal{V}_{j+1}^* \rightarrow  \mathcal{V}_{j+1}, \quad
U_{j+1}^{-1} : \mathcal{V}_{j+1} \rightarrow \mathcal{V}_{j+1}^* 
\end{equation} 
with the following estimates
\begin{equation}\label{estim2iter}
\begin{cases}
|E_{j+1}|_{j+1}^* \MP s_j^{l-\lambda-\tau}, \quad |\nabla_\theta E_{j+1}|_{j+1} \MP s_j^{l-\lambda-\nu}, \quad |\nabla_\omega E_{j+1}|_{j+1} \MP s_j^{l-\lambda-\tau-\nu},\\
|F_{j+1}|_{j+1}^* \MP s_j^{l-\lambda-\nu}, \quad |\nabla_\theta F_{j+1}|_{j+1} \MP s_j^{l-\lambda-\nu-1}, \quad |\nabla_\omega F_{j+1}|_{j+1} \MP s_j^{l-\lambda-2\nu}, \\
|G_{j+1}|_{j+1}^* \MP s_j^{l-\nu}, \quad |\nabla_\theta G_{j+1}|_{j+1} \MP s_j^{l-\nu-1}, \quad |\nabla_\omega G_{j+1}|_{j+1} \MP s_j^{l-2\nu},  \\
|\varphi_{j+1}-\mathrm{Id}|_{j+1} \MP s_j^{l-\lambda}, \quad |\nabla \varphi_{j+1}-\mathrm{Id}|_{j+1} \MP s_j^{l-\lambda-\nu}.
\end{cases}
\end{equation}
\end{proposition}

\begin{proof}
The proof is an induction on $j \in \N$, and we start with the case $j=0$. The Hamiltonian $N+P_0=N_0+P_0$ is defined on $\mathcal{U}_0$, and since $u_0 \geq s_0=\max\{s_0,r_0,h_0\}$, the latter contains $\mathcal{V}_0$, and hence
\[ |P_0|_0=\varepsilon_0\MP \varepsilon. \]
It follows from the definition of $s_0$ in~\eqref{s0} that
\[ |P_0|_0=\varepsilon_0\leq s_0^l. \]
To apply Proposition~\ref{Posj} with $R_0=P_0$, it is sufficient to have
\begin{equation*}
s_0^l \MP s_0^{\lambda+\nu}
\end{equation*}
and this is satisfied for $\varepsilon$, and thus $s_0$, sufficiently small, as $l>\lambda+\nu$. We can therefore apply Proposition~\ref{Posj} to find $\mathcal{F}_1$ and define $\mathcal{F}^1=\mathcal{F}_1$ to have
\[(N+P_0) \circ \mathcal{F}^1=N_0^++P_0^+. \]
We set $N_1=N_0^+$ and $R_1=P_0^+$ and we get from~\eqref{estim1j} with $j=0$ that
\[ |R^1|_1 \leq \delta^ls_0^l/2= s_1^l/2. \]
We have that $\mathcal{F}^1$ maps $\mathcal{V}_1^*$, and thus $\mathcal{V}_1$, into $\mathcal{V}_0$; but since $u_1=s_0$ we have in fact that $\mathcal{V}_0$ is contained in $\mathcal{U}_1$, and thus $\mathcal{F}^1$ maps $\mathcal{V}_1$ into $\mathcal{U}_1$. The estimates~\eqref{estim2iter} follows directly from the estimates~\eqref{estim2j} with $j=0$, taking into account that $\varepsilon_0\leq s_0^l$.

Now assume that for some $j \geq 1$ we have constructed $\mathcal{F}^j$, $N_{j}$ and $R_{j}$ which satisfies~\eqref{estim1iter}. We need to construct $\mathcal{F}_{j+1}$ as in~\eqref{fbas} satisfying~\eqref{estim2iter}, such that $\mathcal{F}^{j+1}=\mathcal{F}^j \circ \mathcal{F}_{j+1}$ is as in~\eqref{fhaut} and satisfies~\eqref{estim1iter}. Let us write
\[ N+P_{j}=N+P_{j-1} +(P_{j} -P_{j-1}) \]
so that
\[ (N+P_{j})\circ \mathcal{F}^j=N_j+R_j+(P_{j} -P_{j-1}) \circ \mathcal{F}^j .\]
By our inductive assumption, $\mathcal{F}^j$ maps $\mathcal{V}_{j}$ into $\mathcal{U}_j$ and hence we have from Proposition~\ref{smoothing}
\[|(P_{j}-P_{j-1}) \circ\mathcal{F}^j|_{j} \leq |P_{j}-P_{j-1}|_{\mathcal{U}_j} \MP u_{j-1}^l\varepsilon \MP \delta^{-2l}s_j^l\varepsilon \leq s_j^l/2 \]
for $\varepsilon$ small enough. Observe also that by the induction hypothesis
\[ |R_{j}|_{j} \leq s_{j}^l/2 \]
so if we set
\[ \hat{R}_j=R_j+(P_{j} -P_{j-1}) \circ \mathcal{F}^j  \]
we arrive at
\[ (N+P_{j})\circ \mathcal{F}^j=N_j+\hat{R}_j, \quad |\hat{R}_j|_{j}=\varepsilon_j \leq s_j^l.  \]
To apply Proposition~\ref{Posj} to this Hamiltonian, it is sufficient to have
\begin{equation*}
s_j^l \MP s_j^{\lambda+\nu}
\end{equation*}
which is satisfied since this is the case for $j=0$. Proposition~\ref{Posj} applies and we find $\mathcal{F}_{j+1}$ as in~\eqref{fbas}, and the estimates~\eqref{estim2iter} follows from~\eqref{estim2j} and the fact that $\varepsilon_j\leq s_j^l$. We may set $N_{j+1}=N_j^+$, $R_{j+1}=\hat{R}_j^+$ and again we get from~\eqref{estim1iter} 
\[ |R_{j+1}|_1 \leq \delta^ls_j^l/2= s_{j+1}^l/2. \]
To complete the induction, the only thing that remains to be checked is~\eqref{fhaut}, that is we need to show that $\mathcal{F}^{j+1}$ maps $\mathcal{V}_{j+1}$ into $\mathcal{U}_{j+1}$. To prove this, we proceed as in~\cite{Pop04} and first observe that for all $0 \leq i \leq j$, letting $W_i=\mathrm{diag}(s_i^{-1}\mathrm{Id},r_i^{-1}\mathrm{Id},h_i^{-1}\mathrm{Id})$, we have from~\eqref{estim2iter} that
\begin{equation}\label{estweight1}
|W_i(\nabla \mathcal{F}_{i+1}-\mathrm{Id})W_i^{-1}|_{i+1} \leq C s_{i}^{l-\lambda-\nu}.
\end{equation} 
for a large constant $C>0$. We also have
\begin{equation}\label{estweight2}
|W_i W_{i+1}^{-1}|=\max\{\delta,\delta^\lambda,\delta^\nu\}=\delta
\end{equation} 
and thus, for $\varepsilon$ small enough, we obtain
\begin{eqnarray}\nonumber\label{co}
|W_0\nabla \mathcal{F}^{j+1}W_{j}^{-1}|_{j+1} & = & |W_0\nabla (\mathcal{F}_1 \circ \cdot \circ \mathcal{F}_{j+1})W_{j}^{-1}| 
\\ \nonumber
& \leq & \prod_{i=0}^{j-1}\left( |W_i W_{i+1}^{-1}| |W_i\nabla \mathcal{F}_{i+1} W_i^{-1}|_{i+1} \right) |W_j\nabla \mathcal{F}_{j+1} W_j^{-1}|_{j+1} \\
& \leq & \delta^j \prod_{i=0}^{+\infty}(1+Cs_{i}^{l-\lambda-\nu}) \leq  2\delta^j  
\end{eqnarray}
for $\varepsilon$ (and thus $s_0$) small enough. Now let us decompose $z=(\theta,I,\omega)=x+iy$ into its real and imaginary part, and write
\begin{equation}\label{comp}
\mathcal{F}^{j+1}(x+iy)=\mathcal{F}^{j+1}(x)+W_0^{-1}T_{j+1}(x,y)W_j y
\end{equation}
where
\[ T_{j+1}(x,y)=i\int_0^1W_0\nabla \mathcal{F}^{j+1}(x+tiy)W_j^{-1}dt.  \]
The important observation is that $\mathcal{F}^{j+1}$ is real-analytic, thus $\mathcal{F}^{j+1}(x)$ is real and it is sufficient to prove that the second term in~\eqref{comp} is bounded by $u_{j+1}=u_0\delta^{j+1}=s_0\delta^j$ when $z \in \mathcal{V}_{j+1}$. But for $z\in \mathcal{V}_{j+1}$, we have $|W_jy| \leq s_0$ and since $|W_0^{-1}|=s_0 \leq 2^{-1}$, we can deduce from~\eqref{co} that
\[ |W_0^{-1}T_{j+1}(x,y)W_j y| \leq  2^{-1}2\delta^js_0=s_0\delta^j=u_{j+1}  \]
which is what we needed to prove.
\end{proof}

The last thing we need for the proof of Theorem~\ref{thm2} is the following  converse approximation result, which we state in a way adapted to our need.

\begin{proposition}\label{convsmoothing}
Let $F^j$ be a sequence of real-analytic functions defined on $\hat{\mathcal{V}}_j=\mathcal{V}_{s_j}$, and which satisfies
\[ F^0=0, \quad |F^{j+1}-F^{j}|_{\hat{\mathcal{V}}_{j+1}} \MP s_j^\alpha, \quad j \in \N \]
for some $\alpha>0$. Then for any $0<\beta \leq \alpha$ which is not an integer, $F \in C^\beta(\T^n)$ and we have
\[ |F|_\beta \MP (\theta(1-\theta))^{-1} s_0^{\alpha-\beta}, \quad \theta=\beta-[\beta]. \]
\end{proposition}

We point out that using the estimates~\eqref{estweight1} and~\eqref{estweight2} one could proceed as in~\cite{Pos01} and easily obtain, on appropriate domains, estimates such as
\begin{equation}\label{weak}
|W_0(\mathcal{F}^{j+1}-\mathcal{F}^{j})| \MP |W_j(\mathcal{F}_{j+1}-\mathrm{Id})|\MP s_j^{l-\lambda-\nu}. 
\end{equation}
However, in view of Proposition~\ref{convsmoothing}, such estimates are not sufficient for our purpose. The point is that the simple argument from~\cite{Pos01} using these weight matrices do not take into account the structure of the transformation (estimating the size of weight matrices as in~\eqref{estweight2} is actually very pessimistic), and the latter will be important for us in the proof of the convergence. This is why we decomposed $\mathcal{F}=(\Phi,\varphi)=(E,F,G,\varphi)$, and we will now proceed as in~\cite{BF17} to prove better estimates than those in~\eqref{weak}.

\begin{proof}[Proof of Theorem~\ref{thm2}]
Let us denote by 
\[ \mathcal{F}^{j+1}=(\Phi^{j+1},\varphi^{j+1})=(U^{j+1},V^{j+1},\varphi^{j+1})=(E^{j+1},F^{j+1},G^{j+1},\varphi^{j+1}) \]
and
\[ \Gamma^{j+1}=G^{j+1} \circ (U^{j+1})^{-1}. \]
We claim that the estimates~\eqref{estim2iter} imply that
\begin{equation}\label{aprouver}
\begin{cases}
 |\varphi^{j+1}-\varphi^j|_{j+1} \MP s_j^{l-\lambda}, \quad |E^{j+1}-E^j|_{j+1} \MP s_j^{l-\lambda-\tau}, \quad |F^{j+1}-F^j|_{j+1} \MP s_j^{l-\lambda-\nu}, \\ 
|G^{j+1}-G^j|_{j+1} \MP s_j^{l-\lambda-\tau}, \quad |\Gamma^{j+1}-\Gamma^j|_{j+1} \MP s_j^{l-\nu}.
\end{cases}
\end{equation}
Let us first assume~\eqref{aprouver}, and show how to conclude the proof. Since the open complex domains $V_{h_j}$ shrink to the closed set $\Omega_{1,\tau}$, the first inequality of~\eqref{aprouver} show that $\varphi^j$ converges to a uniformly continuous map $\varphi : \Omega_{1,\tau} \rightarrow \tilde{\Omega}$.  Then since $l-\lambda-\tau>1$, it follows from the second inequality of~\eqref{aprouver} and Proposition~\ref{convsmoothing} that for a fixed $\omega \in \Omega_{1,\tau}$, $E_\omega^j$ converges to a $C^1$ (in fact $C^r$, for any real non-integer $r$ such that $1<r<l-\lambda-\tau$) map $E_\omega : \T^n \rightarrow B_1$ such that $U_\omega=\mathrm{Id}+E_\omega$ is a diffeomorphism of $\T^n$ and
\begin{equation}\label{est1}
|U_\omega-\mathrm{Id}|_1=|E_\omega|_1 \MP s_0^{l-\lambda-\nu} \MP \varepsilon^{(l-\lambda-\nu)/l}.
\end{equation}
Similarly, since $l-\lambda-\tau>1$ and $l-\nu>\nu=\tau+1$, the last two inequalities of~\eqref{aprouver} and Proposition~\ref{convsmoothing} imply that for a fixed $\omega \in \Omega_{1,\tau}$, $G_\omega^j$ and $\Gamma_\omega^j$ converge respectively to a $C^1$ and $C^{\tau+1}$ maps $G_\omega : \T^n \rightarrow B_1$ and $\Gamma_\omega : \T^n \rightarrow B_1$ such that  
\begin{equation}\label{est2}
|G_\omega|_{1} \MP s_0^{l-\lambda-\nu} \MP \varepsilon^{(l-\lambda-\nu)/l}, \quad |\Gamma_\omega|_{\tau+1} \MP s_0^{l-2\nu} \MP \varepsilon^{(l-2\nu)/l}.
\end{equation}
We can eventually define
\[ \Psi_\omega=(U_\omega,G_\omega) : \T^n \rightarrow \T^n \times B_1 \]
which is the limit of $\Psi_\omega^{j+1}=(U_\omega^{j+1},G_\omega^{j+1})$. On account of~\eqref{estim1iter} we have
\begin{equation}\label{convtrans1}
|H_j \circ \mathcal{F}^{j+1}-N_j|_{j+1} \leq s_{j+1}^l 
\end{equation}
where $H_j=N+P_j$, which, evaluated at $I=0$, implies in particular that
\[ (H_j)_{\varphi^{j+1}(\omega)} \circ \Psi_\omega^{j+1} \]
converges, as $j$ goes to infinity, to a constant. By Proposition~\ref{smoothing}, $P_j$ converges to $P$ in the $C^1$ topology, so in particular $H_j$ converges uniformly to $H$ and thus
\[ \Psi_\omega(\T^n)=\{(\theta,\Gamma_\omega(\theta) \; | \; \theta \in \T^n)\} \]
is an embedded Lagrangian torus invariant by the flow of $H_{\varphi(\omega)}$. Moreover, the inequality~\eqref{convtrans1}, together with Cauchy inequality (using $l>\lambda$) and the symplectic character of $\Phi_\omega^{j+1}$, implies that
\[ (\nabla \Psi_\omega^{j+1})^{-1}X_{(H_j)_{\varphi^{j+1}(\omega)}} \circ \Psi_\omega^{j+1} \] 
converges, as $j$ goes to infinity, to the vector $\omega$. Again, the $C^1$ convergence of $P_j$ to $P$ implies the uniform convergence of $X_{H_j}$ to $X_H$, and the latter means that at the limit we have
\[ X_{H_{\varphi(\omega)}} \circ \Psi_\omega=\nabla \Psi_\omega\cdot \omega. \]
With the estimates~\eqref{est1} and~\eqref{est2}, this proves the first part of the statement.

To prove the second part, we just observe that~\eqref{aprouver} together with a Cauchy estimate gives
\[ |\varphi^{j+1}-\varphi^j|_{j+1} \MP s_j^{l-\lambda-\nu} \PM 1. \]
It follows that $\varphi$ is a limit of uniform Lipschitz functions, and so it is Lipschitz with 
\[ \mathrm{Lip}(\varphi-\mathrm{Id}) \MP s_0^{l-\lambda-\nu} \MP \varepsilon^{(l-\lambda-\nu)/l}.   \]
Similarly, from~\eqref{aprouver} and a Cauchy estimate we have
\[ |\nabla_\omega\Gamma^{j+1} - \nabla _\omega \Gamma^j|_{j+1}^* \MP s_j^{l-2\nu} \PM 1\] 
and thus $\Gamma$ is Lipschitz with respect to $\omega$ with 
\[ \mathrm{Lip}(\Gamma) \MP s_0^{l-2\nu} \MP \varepsilon^{(l-2\nu)/l}.\]  
This gives the second part of Theorem~\ref{thm2}, and now it remains to prove the claim~\eqref{aprouver}. 

To simplify the notations, we will not indicate the domain on which the supremum norms are taken, as this should be clear from the context. Let us denote
\[ \Delta_{j+1}(\theta,\omega)=(U_{j+1}(\theta),\varphi_{j+1}(\omega))=(\theta+E_{j+1}(\theta,\omega),\varphi_{j+1}(\omega)) \]
which, in view of~\eqref{estim2iter}, satisfy
\begin{equation}\label{deltaj}
\begin{cases}
|\Pi_\theta\Delta_{j+1}-\mathrm{Id}|=|E_{j+1}| \MP s_j^{l-\lambda-\tau}, \quad |\nabla_\theta(\Pi_\theta\Delta_{j+1}-\mathrm{Id})| \MP s_j^{l-\lambda-\nu}  \\
|\Pi_\omega\Delta_{j+1}-\mathrm{Id}|=|\varphi_{j+1}-\mathrm{Id}|\MP s_j^{l-\lambda}, \quad |\nabla_\omega(\Pi_\omega\Delta_{j+1}-\mathrm{Id})| \MP s_j^{l-\lambda-\nu}. 
\end{cases}
\end{equation}
Recalling that
\[ \mathcal{F}^{j+1}=(E^{j+1},F^{j+1},G^{j+1},\varphi^{j+1})  \]
is of the form $\mathcal{F}^{j+1}=\mathcal{F}^{j} \circ \mathcal{F}_{j+1}$, with
\[ \mathcal{F}_{j+1}=(\Phi_{j+1},\varphi_{j+1})=(E_{j+1},F_{j+1},G_{j+1},\varphi_{j+1}) \] 
we have the following inductive expressions:
\begin{equation}\label{expressions}
\begin{cases}
\varphi^{j+1}=\varphi^j \circ \varphi_{j+1}, \\
E^{j+1}=E_{j+1}+E^j \circ \Delta_{j+1} \\
F^{j+1}=(\mathrm{Id}+F^j \circ \Delta_{j+1}).F_{j+1}+F^j\circ \Delta_{j+1} \\
G^{j+1}=(\mathrm{Id}+F^j \circ \Delta_{j+1}).G_{j+1}+G^j\circ \Delta_{j+1}.
\end{cases}
\end{equation}
In the sequel, we shall make constant use of the estimates~\eqref{estim2iter}. Let us first prove the estimate for $\varphi^{j+1}$, which is the simplest. A straightforward induction gives
\[ |\nabla \varphi^{j}| \MP \prod_{i=0}^{j-1}(1+s_i^{l-\lambda-\nu}) \MP 1 \] 
and together with 
\[ \varphi^{j+1}-\varphi^{j}=\left(\int_0^1\nabla \varphi^j \circ (t \varphi_{j+1}+(1-t)\mathrm{Id})dt\right)\cdot (\varphi_{j+1}-\mathrm{Id}) \]
one finds
\[ |\varphi^{j+1}-\varphi^{j}| \MP |\nabla \varphi^j| |\varphi_{j+1}-\mathrm{Id}|\MP s_j^{l-\lambda} \]
which is the first estimate of~\eqref{aprouver}. The estimate for $E^{j+1}$ is slightly more complicated. Let us introduce
\[ \hat{\Delta}_{j+1}(\theta,\omega)=(\theta,\varphi_{j+1}(\omega)) \]
and we split
\begin{equation}\label{split}
E^{j+1}-E^{j}=(E^{j+1}-E^j \circ \hat{\Delta}_{j+1})+(E^j \circ \hat{\Delta}_{j+1}-E_j).
\end{equation}
The first summand in~\eqref{split} read
\[ E^{j+1}-E^j \circ \hat{\Delta}_{j+1}=E_{j+1}+E^j \circ \Delta_{j+1}-E^j \circ \hat{\Delta}_{j+1} \]
Using~\eqref{deltaj} one obtains by induction
\begin{equation*}
|\nabla_\theta E^j| \MP \sum_{i=0}^{j-1} s_i^{l-\lambda-\nu} \PM 1
\end{equation*}
and therefore
\[ |E^j \circ \Delta_{j+1}-E^j \circ \hat{\Delta}_{j+1}| \MP |\nabla_\theta E^j||E_{j+1}| \PM s_j^{l-\lambda-\tau} \]
and hence
\begin{equation}\label{E1}
|E^{j+1}-E^j \circ \hat{\Delta}_{j+1}| \MP s_j^{l-\lambda-\tau}. 
\end{equation}
For the second summand of~\eqref{split}, again by induction using~\eqref{deltaj} one obtains
\begin{equation*}
|\nabla_\omega E^j| \MP \sum_{i=0}^{j-1} s_i^{l-\lambda-\tau-\nu} \MP s_{j-1}^{-\tau}\sum_{i=0}^{j-1} s_i^{l-\lambda-\nu} \PM s_{j-1}^{-\tau}
\end{equation*}
since $s_i^{-\tau}\leq s_{j-1}^{-\tau}$ for $0 \leq i \leq j-1$, and hence
\begin{equation}\label{E2}
|E^j \circ \hat{\Delta}_{j+1}-E^j| \MP |\nabla_\omega E^j||\varphi_{j+1}-\mathrm{Id}| \MP s_{j-1}^{-\tau} s_j^{l-\lambda} \PM s_j^{l-\lambda-\tau} . 
\end{equation}
From~\eqref{split},~\eqref{E1} and~\eqref{E2} one arrives at
\[ |E^{j+1} -E^j| \MP s_j^{l-\lambda-\tau} \]
which is the second estimate of~\eqref{aprouver}. The estimate for $F^{j+1}$ is, again, similar to $E^{j+1}$ but more complicated. We use a similar splitting
\begin{equation}\label{split2}
F^{j+1}-F^{j}=(F^{j+1}-F^j \circ \hat{\Delta}_{j+1})+(F^j \circ \hat{\Delta}_{j+1}-F_j)
\end{equation}
and start with the first summand
\begin{equation}
F^{j+1}-F^j \circ \hat{\Delta}_{j+1}=(\mathrm{Id}+F^j \circ \Delta_{j+1})F_{j+1}+F^j \circ \Delta_{j+1}-F^j \circ \hat{\Delta}_{j+1}. 
\end{equation}
To estimate this term, we first prove, by induction using~\eqref{deltaj}, that
\[ |F^j| \MP \sum_{i=0}^{j-1} s_i^{l-\lambda-\nu} \PM 1 \]
which leads to
\[ |\mathrm{Id}+F^j \circ \Delta_{j+1}| \MP 1, \quad |F_{j+1}\nabla_\theta \Delta_{j+1}+\Delta_{j+1}| \PM 1. \]
A computation now gives
\[ \nabla_\theta F^{j+1}=(\mathrm{Id}+F^{j} \circ \Delta_{j+1})\nabla_\theta F_{j+1}+(F_{j+1}\nabla_\theta\Delta_{j+1}+\Delta_{j+1})\nabla_\theta F^j \]
and by induction, we can now claim that
\begin{equation*}
|\nabla_\theta F^j| \MP \sum_{i=0}^{j-1} s_i^{l-\lambda-\nu-1} \PM s_{j-1}^{-1}.
\end{equation*}
Proceeding as before, this leads to
\[ |F^j \circ \Delta_{j+1}-F^j \circ \hat{\Delta}_{j+1}| \MP |\nabla_\theta F^j||E_{j+1}| \MP s_{j-1}^{-1}s_j^{l-\lambda-\tau} \PM s_j^{l-\lambda-\nu} \]
and hence
\begin{equation}\label{F1}
|F^{j+1}-F^j \circ \hat{\Delta}_{j+1}| \MP s_j^{l-\lambda-\nu}. 
\end{equation}
For the second summand, a similar but slightly more involved computation and induction leads to
\begin{equation*}
|\nabla_\omega F^j| \MP \sum_{i=0}^{j-1} s_i^{l-\lambda-2\nu} \PM s_{j-1}^{-\nu}
\end{equation*} 
and hence
\begin{equation}\label{F2}
|F^j \circ \hat{\Delta}_{j+1}-F^j| \MP |\nabla_\omega F^j||\varphi_{j+1}-\mathrm{Id}| \PM s_j^{-\nu} s_j^{l-\lambda} \PM s_j^{l-\lambda-\nu} . 
\end{equation}
From~\eqref{split2},~\eqref{F1} and~\eqref{F2} we obtain
\[ |F^{j+1} -F^j| \MP s_j^{l-\lambda-\nu} \]
which is the third estimate of~\eqref{aprouver}. The estimate for $G^{j+1}$ follows from the exact same argument as the one used for $F^{j+1}$: one proves by induction that
\[ |\nabla_\theta G^j| \MP \sum_{i=0}^{j-1} s_i^{l-\nu-1} \PM 1, \quad |\nabla_\omega G^j| \MP \sum_{i=0}^{j-1} s_i^{\lambda-\nu} \PM 1  \]
which leads to
\begin{equation*}
|G^{j+1}-G^j \circ \hat{\Delta}_{j+1}| \PM s_j^{l-\lambda-\tau}
\end{equation*}
and
\begin{equation*}
|G^j \circ \hat{\Delta}_{j+1}-G^j| \PM s_j^{l-\lambda} \PM s_j^{l-\lambda-\tau} 
\end{equation*}
which combined give the fourth estimate of~\eqref{aprouver}. To prove the last estimate, observe that the inductive expression for $E^{j+1}$ gives
\[ U^{j+1}=U^j \circ \Delta_{j+1} \]
and with the inductive expression for $G^{j+1}$ this leads to
\[ G^{j+1} \circ (U^{j+1})^{-1}=(\mathrm{Id}+F^j \circ \Delta_{j+1}).G_{j+1} \circ (U^{j+1})^{-1}+G^j\circ (U^{j})^{-1} \]
and therefore
\[ \Gamma^{j+1}-\Gamma^j=(\mathrm{Id}+F^j \circ \Delta_{j+1}).G_{j+1} \circ (U^{j+1})^{-1}. \]
Since the image of $(U_{j+1})^{-1}$ is contained in $\mathcal{V}_{j+1}^*$, the same holds true for $(U^{j+1})^{-1}$ and therefore
\[ |G_{j+1} \circ (U^{j+1})^{-1}| \leq |G_{j+1}|^* \MP s_j^{l-\nu} \]
which gives
\[ |\Gamma^{j+1}-\Gamma^j| \MP s_j^{l-\nu}.  \]
This concludes the proof of the claim, and hence finishes the proof of the theorem.
\end{proof}

\medskip

\textit{Acknowledgements.} This material is based upon work supported by the National Science Foundation under Grant No. 1440140, while the author was in residence at the Mathematical Sciences Research Institute in Berkeley, California, during the thematic program ``Hamiltonian systems, from topology to applications through analysis". It is a pleasure for me to thank Jacques F{\'e}joz for many fruitful conversations. I have also benefited from partial funding from the ANR project Beyond KAM. 

\addcontentsline{toc}{section}{References}
\bibliographystyle{amsalpha}
\bibliography{KAMmesCl}

\end{document}